\documentclass[10pt]{amsart}
\usepackage{amsmath,amscd}
\usepackage{amssymb}

\usepackage[all]{xy}

\usepackage[T1]{fontenc}

\usepackage{times}

\author{Liran Shaul}
\address{Universiteit Antwerpen, Departement Wiskunde-Informatica, Middelheim campus,
Middelheimlaan 1,
2020 Antwerp, Belgium}
\email{Liran.Shaul@uantwerpen.be}
\newtheorem{thm}[equation]{Theorem}
\newtheorem{cor}[equation]{Corollary}
\newtheorem{prop}[equation]{Proposition}

\theoremstyle{definition}
\newtheorem{dfn}[equation]{Definition}
\newtheorem{rem}[equation]{Remark}
\newtheorem{exa}[equation]{Example}

\newtheorem{setup}[equation]{Setup}

\newcommand{\inj}{\hookrightarrow}
\newcommand{\surj}{\twoheadrightarrow}

\newcommand{\opn}{\operatorname}
\newcommand{\mfrak}[1]{\mathfrak{#1}}

\newcommand{\mrm}[1]{\mathrm{#1}}

\renewcommand{\k}{\Bbbk}

\renewcommand{\a}{\mfrak{a}}


\begin{document}
\title{Relations between derived Hochschild functors via twisting}
 
\begin{abstract}
Let $\k$ be a regular ring, and let $A,B$ be essentially finite type $\k$-algebras. For any functor $F:\mrm{D}(\opn{Mod}A)\times\dots\times\mrm{D}(\opn{Mod}A)\to\mrm{D}(\opn{Mod}B)$ between their derived categories, we define its twist $F^{!}:\mrm{D}(\opn{Mod}A)\times\dots\times\mrm{D}(\opn{Mod}A)\to\mrm{D}(\opn{Mod}B)$ with respect to dualizing complexes, generalizing Grothendieck's construction of $f^{!}$. We show that relations between functors are preserved between their twists, and deduce that various relations hold between derived Hochschild (co)-homology and the $f^{!}$ functor. We also deduce that the set of isomorphism classes of dualizing complexes over a ring (or a scheme) form a group with respect to derived Hochschild cohomology, and that the twisted inverse image functor is a group homomorphism.
\end{abstract}

\maketitle

\section{Introduction}

All rings in this note are commutative. Let $\k$ be a regular noetherian ring of finite Krull dimension. Let $f:A\to B$ be a map between two essentially finite type $\k$-algebras. Grothendieck duality theory, whose details first appeared in \cite{RD}, centers around the twisted inverse image functor $f^{!}:\mrm{D}^{+}_{\mrm{f}}(\opn{Mod} A) \to \mrm{D}^{+}_{\mrm{f}}(\opn{Mod} B)$ (See \cite{Li} for a modern account). Under the above assumption on $\k$, this functor may be constructed as a twist of the inverse image functor $\mrm{L}f^{*} (-) :=  B\otimes^{\mrm{L}}_A -$. The twist is given by $f^{!}(-) := D_B ( \mrm{L}f^{*}( D_A(-)) )$ where for an essentially finite type $\k$-algebra $g:\k \to C$, we have set $D_C(-) := \mrm{R}\opn{Hom}_C(-,R_C)$, where $R_C = g^{!}(\k)$ is the canonical (or rigid) dualizing complex over $C$ relative to $\k$. Similarly to this construction, given any functor $F$ from $\mrm{D}_{\mrm{f}}(\opn{Mod} A)$ to $\mrm{D}_{\mrm{f}}(\opn{Mod} B)$, one may construct the twist of $F$ by declaring $F^!(-) := D_B (F(D_A(- ))$.  Under suitable finiteness assumptions, if $F,G$ and $H$ are three functors of this form, and if $F \cong G\circ H$, then it is easy to see that $F^{!} \cong G^{!} \circ H^{!}$. This means that relations between such functors give rise to relations between their twists. 

If the ring $A$ is projective over $\k$, then the Hochschild cohomology functor of $A$ over $\k$ is defined by $\opn{Ext}^{*}_{A\otimes_{\k} A}(A,-)$. When dropping the projectivity assumption, an important variation of this construction is given by Shukla cohomology, also known as derived Hochschild cohomology. This functor, recently studied in great detail in \cite{AILN}, is defined by the formula
\[
\mrm{R}\opn{Hom}_{A\otimes^{\mrm{L}}_{\k} A}(A,-)
\]
where the derived tensor product $A\otimes^{\mrm{L}}_{\k} A$ is taken in the category of DG-algebras. Taking the coefficients complex to be of the form $M \otimes^{\mrm{L}}_{\k} N$ where $M,N \in \mrm{D}(\opn{Mod} A)$, it was shown in \cite[Theorem 4.1]{AILN}, under suitable technical assumptions, that this functor has a particularly nice reduction formula: There is a bifunctorial isomorphism 
\begin{equation}\label{eqn:red-hoch}
\mrm{R}\opn{Hom}_{A\otimes^{\mrm{L}}_{\k} A}(A,M \otimes^{\mrm{L}}_{\k} N) \cong \mrm{R}\opn{Hom}_A(\mrm{R}\opn{Hom}_A(M,R_A),N).
\end{equation}
We will see below that the right hand side of this formula is canonically isomorphic to the twist of the bifunctor $-\otimes^{\mrm{L}}_A -$. This suggests the notation $-\otimes^{!}_A -$ for this functor. A similar result (\cite[Theorem 4.6]{AILN}) was given for derived Hochschild homology, and, similarly, we interpret this result by showing that it is canonically isomorphic to the twist of the bifunctor $\mrm{R}\opn{Hom}_A(-,-)$, which suggests the notation $\opn{Hom}^{!}_A(-,-)$ for this functor. 

Thus, having identified the twists of the inverse image functor, the derived tensor functor and the derived hom functor, we will immediately deduce various relations that hold between the twisted inverse image, the derived Hochschild homology and the derived Hochschild cohomology functors. 

\section{Twisted functors}

\begin{setup}
Throughout this section, $\k$ will be a fixed regular noetherian ring of finite Krull dimension.
\end{setup}

Let $A$ be an essentially finite type $\k$-algebra. This means that the structure map $\k \to A$ may be factored as
\[
\k \inj \k[x_1,\dots,x_n] \inj U^{-1}\k[x_1,\dots,x_n] \surj A
\]
where $U \subseteq \k[x_1,\dots,x_n]$ is some multiplicatively closed set. In the category of $\k$-algebras, the canonical dualizing complex of $A$ is given by
\[
R_A := \mrm{R}\opn{Hom}_{U^{-1}\k[x_1,\dots,x_n]}(A, \Omega^n_{U^{-1}\k[x_1,\dots,x_n]/\k} [n]),
\]
where $\Omega^n$ is the module of K\"ahler $n$-differentials. This construction is independent of the chosen factorization (see for example \cite[Theorem 1.1]{AILN}).

An intrinsic characterization of the complex $R_A$ is given by the important property of rigidity: recall that according to \cite[Definition 2.1]{YZ1}, a pair $(R,\rho)$ where $R\in \mrm{D}^{\mrm{b}}_{\mrm{f}}(\opn{Mod} A)$, and 
\[
\rho: R \to \mrm{R}\opn{Hom}_{A\otimes^{\mrm{L}}_{\k} A}(A,R\otimes^{\mrm{L}}_{\k} R)
\]
is an isomorphism, is called a rigid complex over $A$ relative to $\k$. If moreover the complex $R$ is a dualizing complex then $(R,\rho)$ is called a rigid dualizing complex over $A$ relative to $\k$. This notion originated in \cite{VdB}. It is shown in \cite[Theorem 3.6]{YZ1} that a rigid dualizing complex exists, and is unique in a strong sense (see also \cite[Theorem 8.5.6]{AIL1} for a stronger existence result). It is given by the complex $R_A$ defined above.

In the remaining of this note, for any essentially finite type $\k$-algebra $A$, we will denote by $R_A$ the canonical (=rigid) dualizing complex of $A$, and by $D_A$ the functor $D_A(M) := \mrm{R}\opn{Hom}_A(M,R_A)$. In a somewhat unorthodox manner, we will also set for $n>1$, 
\[
D_A(M_1,\dots,M_n):=(D_A(M_1),\dots,D_A(M_n)).
\]
For a category $\mathcal{A}$, we will denote by $\mathcal{A}^n$ the product category $\underbrace{\mathcal{A}\times \dots \times \mathcal{A}}_n$.
	
Generalizing Grothendieck's construction of the twisted inverse image functor leads us to the following definition:

\begin{dfn}\label{def:basic}
Let $A,B$ be two essentially finite type $\k$-algebras, and let $F:\mrm{D}(\opn{Mod} A)^n  \to \mrm{D}(\opn{Mod} B)^m$ be a functor. The twist of $F$ is the functor
\[
F^{!}(-) := D_B \circ F \circ D_A (-) :\mrm{D}(\opn{Mod} A)^n  \to \mrm{D}(\opn{Mod} B)^m.
\]
\end{dfn}

We now give several examples to demonstrate the usefulness of this definition.

\begin{exa}
Let $A,B$ be two essentially finite type $\k$-algebras, and let $f:A\to B$ be a $\k$-algebra map. Consider the functor $\mrm{L}f^{*} : \mrm{D}^{+}_{\mrm{f}}(\opn{Mod} A) \to \mrm{D}^{+}_{\mrm{f}}(\opn{Mod} B)$ given by $\mrm{L}f^{*} (-) := -\otimes^{\mrm{L}}_A B$. Then by \cite[Theorem 4.10]{YZ1}, for any $M \in \mrm{D}^{+}_{\mrm{f}}(\opn{Mod} A)$, there is an isomorphism of functors
\[
(\mrm{L}f^{*})^{!} (M) \cong f^{!}(M).
\]
\end{exa}

\begin{exa}
Let $A$ be an essentially finite type $\k$-algebra, and let $F(M,N) := M\otimes^{\mrm{L}}_A N$ and $G(M,N) := \mrm{R}\opn{Hom}_A(M,N)$.
Then by definition 
\[
F^{!}(M,N) := \mrm{R}\opn{Hom}_A( \mrm{R}\opn{Hom}_A(M,R_A) \otimes^{\mrm{L}}_A \mrm{R}\opn{Hom}_A(N,R_A) ,R_A),
\]
and
\[
G^{!}(M,N) := \mrm{R}\opn{Hom}_A(  \mrm{R}\opn{Hom}_A(\mrm{R}\opn{Hom}_A(M,R_A),\mrm{R}\opn{Hom}_A(N,R_A)) ,R_A).
\]
We set $M\otimes^{!}_A N:= F^{!}(M,N)$, and $\opn{Hom}^{!}_A(M,N) := G^{!}(M,N)$.
Note that if $M,N \in \mrm{D}^{+}_{\mrm{f}}(\opn{Mod} A)$, then $M\otimes^{!}_A N \in \mrm{D}^{+}_{\mrm{f}}(\opn{Mod} A)$. Similarly, if $M \in \mrm{D}^{+}_{\mrm{f}}(\opn{Mod} A)$ and $N \in \mrm{D}^{-}_{\mrm{f}}(\opn{Mod} A)$ then $\opn{Hom}^{!}_A(M,N) \in \mrm{D}^{-}_{\mrm{f}}(\opn{Mod} A)$.
In Theorem \ref{thm:hocharetwist} below, which follows almost immediately from the results of \cite{AILN}, we will identify these two functors.
\end{exa}

\begin{exa}
Let $A$ be an essentially finite type $\k$-algebra, and let $\a \subseteq A$ be an ideal. The $\a$-torsion and $\a$-completion functors are defined by $\Gamma_{\a} (-) := \varinjlim \opn{Hom}_A(A/\a^n,-)$ and $\Lambda_{\a} (-) := \varprojlim A/\a^n \otimes_A -$ respectively. Their derived functors $\mrm{R}\Gamma_{\a}$ and $\mrm{L}\Lambda_{\a}$ exist, and are calculated using K-injective and K-flat resolutions respectively (see \cite[Section 1]{AJL}). It follows from the Greenlees-May duality (specifically, from \cite[Corollary 5.2.2]{AJL}) that for any $M\in \mrm{D}_{\mrm{f}}(\opn{Mod} A)$, there is an isomorphism of functors $(\mrm{R}\Gamma_{\a})^{!}(M) \cong \mrm{L}\Lambda_{\a}(M)$, and for any $M\in \mrm{D}(\opn{Mod} A)$ such that $\mrm{R}\Gamma_{\a}(M) \in \mrm{D}_{\mrm{f}}(\opn{Mod} A)$, there is an isomorphism of functors $(\mrm{L}\Lambda_{\a})^{!}(M) \cong \mrm{R}\Gamma_{\a}(M)$. 
Conversely, once one knows that the functors $\mrm{R}\Gamma_{\a}$ and $\mrm{L}\Lambda_{\a}$ are twists of each other, the Greenlees-May duality
\[
\mrm{R}\opn{Hom}_A(\mrm{R}\Gamma_{\a}(M),N) \cong \mrm{R}\opn{Hom}_A(M,\mrm{L}\Lambda_{\a}(N))
\]
follows easily, so the language of Definition \ref{def:basic} is suited to describe this fundamental relation between these functors.
\end{exa}

\begin{thm}\label{thm:hocharetwist}
Let $A$ be an essentially finite type $\k$-algebra. 
\begin{enumerate}
\item For any $M \in \mrm{D}^{\mrm{b}}_{\mrm{f}}(\opn{Mod} A)$, and any $N \in \mrm{D}_{\mrm{f}}(\opn{Mod} A)$, there is an isomorphism of functors
\[
M\otimes^{!}_A N \cong \mrm{R}\opn{Hom}_{A\otimes^{\mrm{L}}_{\k} A}(A,M\otimes^{\mrm{L}}_{\k} N).
\]
\item  For any $M,N \in \mrm{D}^{\mrm{b}}_{\mrm{f}}(\opn{Mod} A)$, there is an isomorphism of functors
\[
\opn{Hom}^{!}_A(M,N) \cong A \otimes^{\mrm{L}}_{A\otimes^{\mrm{L}}_{\k} A} \mrm{R}\opn{Hom}_{\k}(M,N)
\]
\end{enumerate}
\end{thm}
\begin{proof}
For the first claim, by \cite[Theorem 4.1]{AILN}, there is an isomorphism of functors
\[
\mrm{R}\opn{Hom}_{A\otimes^{\mrm{L}}_{\k} A}(A,M\otimes^{\mrm{L}}_{\k} N) \cong \mrm{R}\opn{Hom}_A(D_A(M),N).
\]
Since $N \in \mrm{D}_{\mrm{f}}(\opn{Mod} A)$, we have that $N \cong D_A(D_A(N))$. Hence, by the derived hom-tensor adjunction
\begin{eqnarray}
\mrm{R}\opn{Hom}_A(D_A(M),N) \cong \mrm{R}\opn{Hom}_A(D_A(M),\mrm{R}\opn{Hom}_A(D_A(N),R_A)) \cong\nonumber\\
\mrm{R}\opn{Hom}_A(D_A(M) \otimes^{\mrm{L}}_A D_A(N) ,R_A),\nonumber
\end{eqnarray}
which proves the result.
To show the second claim, by \cite[Theorem 4.6]{AILN}, there is an isomorphism of functors
\[
A \otimes^{\mrm{L}}_{A\otimes^{\mrm{L}}_{\k} A} \mrm{R}\opn{Hom}_{\k}(M,N) \cong D_A(M) \otimes^{\mrm{L}}_A N.
\]
Since $D_A(M) \otimes^{\mrm{L}}_A N \in \mrm{D}_{\mrm{f}}(\opn{Mod} A)$, we have that 
\[
D_A(M) \otimes^{\mrm{L}}_A N \cong D_A( \mrm{R}\opn{Hom}_A( D_A(M) \otimes^{\mrm{L}}_A N, R_A) ).
\]
So by the derived hom-tensor adjunction
\[
D_A( \mrm{R}\opn{Hom}_A( D_A(M) \otimes^{\mrm{L}}_A N, R_A) ) \cong
D_A( \mrm{R}\opn{Hom}_A( D_A(M) , D_A(N) ) ) = \opn{Hom}^{!}_A(M,N).
\]
\end{proof}

The main reason we introduced the twisting formalism is that relations between functors are preserved between their twists:

\begin{prop}\label{keyprop}
Let $A,B,C$ be three essentially finite type $\k$-algebras. Let $F:\mrm{D}(\opn{Mod}A)^m \to \mrm{D}(\opn{Mod} C)^k$, $G:\mrm{D}(\opn{Mod} B)^n \to \mrm{D}(\opn{Mod} C)^k$ and $H:\mrm{D}(\opn{Mod} A)^m \to \mrm{D}(\opn{Mod} B)^n$ be three functors such that there is an isomorphism of functors $F \cong G\circ H$. Then there is an isomorphism of functors $F^{!}(M_1,\dots,M_m) \cong G^{!} \circ H^{!}(M_1,\dots,M_m)$ for any $M_1,\dots,M_m \in D(\opn{Mod} A)$ such that $H \circ D_A(M_1,\dots,M_m) \in D_{\mrm{f}}(\opn{Mod} B)^n$.
\end{prop}
\begin{proof}
Since $F \cong G\circ H$, it follows that $F^{!} \cong (G\circ H)^! := D_C \circ G\circ H \circ D_A$. On the other hand, by definition 
\[
G^{!}\circ H^{!} := D_C \circ G \circ D_B \circ D_B \circ H \circ D_A.
\]
By assumption, $H \circ D_A$ has finitely generated cohomology. Hence,
$D_B \circ D_B \circ H \circ D_A \cong H\circ D_A$ which proves the result.
\end{proof}

From this proposition, the following relations between twisted functors follow immediately:

\begin{cor}\label{cor:rel}
Let $A$ be an essentially finite type $\k$-algebra. Then the following holds:
\begin{enumerate}
\item Let $B$ be another essentially finite type $\k$-algebra, and let $f:A\to B$ be a $\k$-algebra map. For any $M,N \in \mrm{D}^{+}_{\mrm{f}}(\opn{Mod} A)$ there is a bifunctorial isomorphism
\[
f^{!}(M\otimes^{!}_A N) \cong (f^{!}(M))\otimes^{!}_B (f^{!}(N)).
\]
in $\mrm{D}(\opn{Mod} B)$.
\item For any $M,N,K \in \mrm{D}^{+}_{\mrm{f}}(\opn{Mod} A)$, there is a trifunctorial isomorphism
\[
M\otimes^{!}_A (N \otimes^{!}_A K) \cong (M\otimes^{!}_A N) \otimes^{!}_A K 
\]
in $\mrm{D}(\opn{Mod} A)$.
\item  For any $M\in \mrm{D}^{+}_{\mrm{f}}(\opn{Mod} A)$, 
$N \in \mrm{D}^{\mrm{b}}_{\mrm{f}}(\opn{Mod} A)$ and $K \in \mrm{D}^{-}_{\mrm{f}}(\opn{Mod} A)$, there is a trifunctorial isomorphism
\[
\opn{Hom}^{!}_A(M\otimes^{!}_A N,K) \cong \opn{Hom}^{!}_A(M,\opn{Hom}^{!}_A(N,K))
\]
in $\mrm{D}(\opn{Mod} A)$.
\end{enumerate}
\end{cor}
\begin{proof}
Each of these statements follows from applying Proposition \ref{keyprop} to the following canonical isomorphisms:
\begin{enumerate}
\item $(M\otimes^{\mrm{L}}_A N) \otimes^{\mrm{L}}_A B  \cong (M\otimes^{\mrm{L}}_A  B) \otimes^{\mrm{L}}_B (N \otimes^{\mrm{L}}_A B)$.
\item $M \otimes^{\mrm{L}}_A (N \otimes^{\mrm{L}}_A K) \cong (M \otimes^{\mrm{L}}_A  N) \otimes^{\mrm{L}}_A  K$.
\item $\mrm{R}\opn{Hom}_A(M \otimes^{\mrm{L}}_A N,K) \cong \mrm{R}\opn{Hom}_A(M,\mrm{R}\opn{Hom}_A(N,K))$.
\end{enumerate}
\end{proof}

\begin{rem}\label{rem:monoidal}
It follows easily from this corollary that the operation $-\otimes^{!}_A-$ defines a symmetric monoidal structure on $\mrm{D}^{+}_{\mrm{f}}(\opn{Mod} A)$ for any essentially finite type $\k$-algebra $A$. The canonical dualizing complex $R_A$ is a monoidal unit. If $B$ is another essentially finite type $\k$-algebra, and $f:A\to B$ is a $\k$-algebra map, then the twisted inverse image functor $f^{!}:\mrm{D}^{+}_{\mrm{f}}(\opn{Mod} A) \to \mrm{D}^{+}_{\mrm{f}}(\opn{Mod} B)$ is a monoidal functor. Further, restricting attention to the subcategory of $\mrm{D}^{+}_{\mrm{f}}(\opn{Mod} A)$ made of complexes with finitely generated cohomology and of finite injective dimension over $A$, one obtain a closed symmetric monoidal category, with the internal hom being $\opn{Hom}^{!}_A(-,-)$. 
\end{rem}

\begin{rem}
We first encountered the idea that Hochschild cohomology defines a symmetric monoidal structure in \cite{Ga}.
There, in \cite[Corollary 5.6.8]{Ga}, assuming $\k$ is a field of characteristic zero, it was stated without proof that the operation $\mrm{R}\opn{Hom}_{A\otimes_{\k} A}(A,-\otimes_{\k} -)$ defines a symmetric monoidal structure on the category of indcoherent sheaves on $\opn{Spec} A$.
\end{rem}

Combining Theorem \ref{thm:hocharetwist} and Corollary \ref{cor:rel} we immediately obtain various relations between the derived Hochschild functors and the twisted inverse image functor.

\begin{cor}\label{cor:commute}
\textbf{Derived Hochschild cohomology commutes with the twisted inverse image functor:}
Let $\k$ be a regular noetherian ring of finite Krull dimension, and let $A,B$ be two essentially finite type $\k$-algebras. Let $f:A\to B$ be a $\k$-algebra map. Let $M,N \in \mrm{D}^{\mrm{b}}_{\mrm{f}}(\opn{Mod} A)$ and assume that the complexes $f^{!}(M), f^{!}(N)$ have bounded cohomology.
Then there is a bifunctorial isomorphism
\[
f^{!} \mrm{R}\opn{Hom}_{A\otimes^{\mrm{L}}_{\k} A}(A,M\otimes^{\mrm{L}}_{\k} N) \cong 
\mrm{R}\opn{Hom}_{B\otimes^{\mrm{L}}_{\k} B}(B,f^{!}(M)\otimes^{\mrm{L}}_{\k} f^{!}(N))
\]
in $\mrm{D}(\opn{Mod} B)$.
\end{cor}

\begin{rem}
If in the above corollary the map $f:A\to B$ has a finite flat dimension, then by \cite[Proposition 2.5.4]{AIL2}, assuming that $M,N$ have a bounded cohomology implies that $f^{!}(M),f^{!}(N)$ have bounded cohomology.
\end{rem}

\begin{cor}
\textbf{Adjunction between derived Hochschild homology and derived Hochschild cohomology:}
Let $\k$ be a regular noetherian ring of finite Krull dimension, and let $A$ be an essentially finite type $\k$-algebra. Let $M,N,K \in \mrm{D}^{\mrm{b}}_{\mrm{f}}(\opn{Mod} A)$ be three complexes, and assume that the complexes $M\otimes^{!}_A N, \opn{Hom}^{!}_A(N,K)$ are also bounded. Then there is a trifunctorial isomorphism
\begin{eqnarray}
\nonumber
A \otimes^{\mrm{L}}_{A\otimes^{\mrm{L}}_{\k} A} \mrm{R}\opn{Hom}_{\k} ( \mrm{R}\opn{Hom}_{A\otimes^{\mrm{L}}_{\k} A}(A,M\otimes^{\mrm{L}}_{\k} N) ,K ) \cong\\
\nonumber 
A \otimes^{\mrm{L}}_{A\otimes^{\mrm{L}}_{\k} A} \mrm{R}\opn{Hom}_{\k} (M ,A \otimes^{\mrm{L}}_{A\otimes^{\mrm{L}}_{\k} A} \mrm{R}\opn{Hom}_{\k} (N,K)  ).
\end{eqnarray}
in $\mrm{D}(\opn{Mod} A)$.
\end{cor}

\begin{cor}\label{cor:ass}
\textbf{Associativity of derived Hochschild cohomology:} 
Let $\k$ be a regular noetherian ring of finite Krull dimension, and let $A$ be an essentially finite type $\k$-algebra. Let $M,N,K \in \mrm{D}^{\mrm{b}}_{\mrm{f}}(\opn{Mod} A)$ be three complexes, and assume that the complexes $M\otimes^{!}_A N$ and $N\otimes^{!}_A K$ are also bounded. Then there are trifunctorial isomorphisms
\begin{eqnarray}
\nonumber
\mrm{R}\opn{Hom}_{A\otimes^{\mrm{L}}_{\k} A}(A,M\otimes^{\mrm{L}}_{\k} \mrm{R}\opn{Hom}_{A\otimes^{\mrm{L}}_{\k} A}(A,N\otimes^{\mrm{L}}_{\k} K)) \cong\\
\nonumber
\mrm{R}\opn{Hom}_{A\otimes^{\mrm{L}}_{\k}  A}(A,\mrm{R}\opn{Hom}_{A\otimes^{\mrm{L}}_{\k} A}(A,M\otimes^{\mrm{L}}_{\k} N) \otimes^{\mrm{L}}_{\k}  K) \cong\\ 
\nonumber
\mrm{R}\opn{Hom}_A ( \mrm{R}\opn{Hom}_A(M,R_A) \otimes^{\mrm{L}}_A \mrm{R}\opn{Hom}_A(N,R_A) \otimes^{\mrm{L}}_A
\mrm{R}\opn{Hom}_A(K,R_A) ,R_A)
\end{eqnarray}
in $\mrm{D}(\opn{Mod} A)$.
\end{cor}
\begin{proof}
The first isomorphism follows from Theorem \ref{thm:hocharetwist}
and Corollary \ref{cor:rel}. To get the second isomorphism, first replace
$\mrm{R}\opn{Hom}_{A\otimes^{\mrm{L}}_{\k}  A}(A,\mrm{R}\opn{Hom}_{A\otimes^{\mrm{L}}_{\k} A}(A,M\otimes^{\mrm{L}}_{\k} N) \otimes^{\mrm{L}}_{\k}  K)$ with $(M\otimes^{!}_A N)\otimes^{!}_A K$, and now use the derived hom-tensor adjunction.
\end{proof}

The second isomorphism in the above Corollary can be thought of as a reduction formula for derived 3-Hochschild cohomology. One might wonder if this functor is canonically isomorphic to $\mrm{R}\opn{Hom}_{A\otimes^{\mrm{L}}_{\k} A\otimes^{\mrm{L}}_{\k} A}(A,M\otimes^{\mrm{L}}_{\k} N \otimes^{\mrm{L}}_{\k} K)$. The answer to this is positive, and in fact, more generally, for every $n>1$, there is an isomorphism of functors
\[
\mrm{R}\opn{Hom}_{\underbrace{A\otimes^{\mrm{L}}_{\k} A \otimes^{\mrm{L}}_{\k} \dots \otimes^{\mrm{L}}_{\k} A}_n}(A,M_1 \otimes^{\mrm{L}}_{\k} M_2 \otimes^{\mrm{L}}_{\k} \dots \otimes^{\mrm{L}}_{\k} M_n) \cong M_1 \otimes^!_A \dots \otimes^!_A M_n
\]
For every complexes of $A$-modules $M_1,\dots M_n$ which satisfy suitable finiteness conditions. Proof of this fact will appear elsewhere (the case $n=2$ is, as seen above, just a rephrase of \cite[Theorem 4.1]{AILN}). In the special case where $n=4$,  this follows easily, using Corollary \ref{cor:commute}, under an additional flatness hypothesis:

\begin{cor}
Let $\k$ be a regular noetherian ring of finite Krull dimension, and let $A$ be a flat essentially of finite type $\k$-algebra. Let $M_1,M_2,M_3,M_4 \in \mrm{D}^{\mrm{b}}_{\mrm{f}}(\opn{Mod} A)$ be four complexes. Assume that $M_1\otimes^{!}_A M_2, M_3\otimes^{!}_A M_4$ are also bounded. Then there is a quad-functorial isomorphism
\begin{eqnarray}
\nonumber
\mrm{R}\opn{Hom}_{A\otimes_{\k} A\otimes_{\k} A \otimes_{\k} A} (A,M_1 \otimes^{\mrm{L}}_{\k} M_2 \otimes^{\mrm{L}}_{\k} M_3 \otimes^{\mrm{L}}_{\k} M_4) \cong\\
\nonumber
\mrm{R}\opn{Hom}_A ( \mrm{R}\opn{Hom}_A(M_1,R_A) \otimes^{\mrm{L}}_A \mrm{R}\opn{Hom}_A(M_2,R_A) \otimes^{\mrm{L}}_A
\\\nonumber
\mrm{R}\opn{Hom}_A(M_3,R_A) \otimes^{\mrm{L}}_A \mrm{R}\opn{Hom}_A(M_4,R_A) ,R_A).
\end{eqnarray}
Hence, under the above hypothesis, the quad-functor $\mrm{R}\opn{Hom}_{A\otimes_{\k} A\otimes_{\k} A \otimes_{\k} A} (A,- \otimes^{\mrm{L}}_{\k} - \otimes^{\mrm{L}}_{\k} - \otimes^{\mrm{L}}_{\k} -)$ is canonically isomorphic to the twisting of the functor $-\otimes^{\mrm{L}}_A-\otimes^{\mrm{L}}_A-\otimes^{\mrm{L}}_A - :\mrm{D}(\opn{Mod} A)^4 \to \mrm{D}(\opn{Mod} A)$.
\end{cor}
\begin{proof}
Let $C=A\otimes_{\k} A$, and let $\Delta: C\to A$ be the diagonal map. Then by Corollary \ref{cor:commute}, there is a natural isomorphism
\[
\Delta^{!} ( (M_1\otimes^{\mrm{L}}_{\k} M_2) \otimes^{!}_C (M_3\otimes^{\mrm{L}}_{\k} M_4) ) \cong \Delta^{!}(M_1\otimes^{\mrm{L}}_{\k} M_2) \otimes^{!}_A \Delta^{!}(M_3\otimes^{\mrm{L}}_{\k} M_4).
\]
Since $\Delta$ is a finite map, $\Delta^{!}(-) \cong \mrm{R}\opn{Hom}_C(A,-)$.
By Theorem \ref{thm:hocharetwist}, the left hand side is canonically isomorphic to
\[
\mrm{R}\opn{Hom}_{A\otimes_{\k} A}(A, \mrm{R}\opn{Hom}_{A\otimes_{\k} A \otimes_{\k} A\otimes_{\k} A}(A\otimes_{\k} A, (M_1\otimes^{\mrm{L}}_{\k} M_2) \otimes^{\mrm{L}}_{\k} (M_3\otimes^{\mrm{L}}_{\k} M_4) ) )
\]
and by the derived hom-tensor adjunction this is canonically isomorphic to
\[
\mrm{R}\opn{Hom}_{A\otimes_{\k} A \otimes_{\k} A\otimes_{\k} A}(A, (M_1\otimes^{\mrm{L}}_{\k} M_2) \otimes^{\mrm{L}}_{\k} (M_3\otimes^{\mrm{L}}_{\k} M_4) ).
\]
Applying Theorem \ref{thm:hocharetwist} to the right hand side, we obtain:
\[
D_A(
D_A (D_A (D_A(M_1)\otimes^{\mrm{L}}_A D_A(M_2)))
\otimes^{\mrm{L}}_A
D_A (D_A (D_A(M_3)\otimes^{\mrm{L}}_A D_A(M_4)))
).
\]
The result now follows from the fact that $D_A \circ D_A \cong 1$ on $D_A(M_1)\otimes^{\mrm{L}}_A D_A(M_2)$ and on $D_A(M_3)\otimes^{\mrm{L}}_A D_A(M_4)$.
\end{proof}

In the remaining of this note, we analyze how the above twisting operations act on dualizing complexes. It has been known long ago (see \cite[Theorem V.3.1]{RD}) that the set of isomorphism classes of dualizing complexes over a commutative ring (at least if $\opn{Spec} A$ is connected) is classified by the Picard group of $A$ and the integers. Moreover, this set is isomorphic to a set which naturally carries a group structure, namely, the group of isomorphism classes of tilting complexes (See \cite{Ye} for details). Let $A$ be a commutative noetherian ring. Recall that a complex $P \in \mrm{D}(\opn{Mod} A)$ is called a tilting complex if there exist a complex $Q \in \mrm{D}(\opn{Mod} A)$ such that $P\otimes^{\mrm{L}}_A Q \cong A$. If $R_1,R_2$ are two dualizing complexes over $A$ then $P = \mrm{R}\opn{Hom}_A(R_1,R_2)$ is a tilting complex, and there is an isomorphism $R_2 \cong R_1 \otimes^{\mrm{L}}_A P$. The set of isomorphism classes of tilting complexes under the derived tensor product operation form an abelian group, called the derived Picard group of $A$ and denoted by $\opn{DPic}(A)$. Given any dualizing complex $R$, the duality operation $\mrm{R}\opn{Hom}_A(-,R)$ defines a bijection between the set of isomorphism classes of dualizing complexes and the set of isomorphism classes of tilting complexes.

In the next result, we observe, using the above formalism, that the set of isomorphism classes of dualizing complexes also carries naturaly a group structure, with the operation being the functor underlying derived Hochschild cohomology (that is, Shukla cohomology). 

\begin{thm}
Let $\k$ be a regular noetherian ring of finite Krull dimension, and let $A$ be an essentially finite type $\k$-algebra. Then the set $\mathcal{D}_A$ of isomorphism classes of dualizing complexes over $A$ form an abelian group with respect to the operation $\mrm{R}\opn{Hom}_{A\otimes^{\mrm{L}}_{\k} A}(A,-\otimes^{\mrm{L}}_{\k} -)$. The inverse of a dualizing complex $R$ is given by $R' := A\otimes^{\mrm{L}}_{A\otimes^{\mrm{L}}_{\k} A} \mrm{R}\opn{Hom}_{\k}(R,R_A)$. The rigid dualizing complex $R_A$ is the identity of the group. The map 
$\mrm{R}\opn{Hom}_A(-,R_A)$ is a group isomorphism between $\mathcal{D}_A$ and $\opn{DPic}(A)$. If $B$ is another essentially finite type $\k$ algebra, and $f:A\to B$ is a $\k$-algebra map then $f^{!}:\mathcal{D}_A \to \mathcal{D}_B$ is a group homomorphism.
\end{thm}
\begin{proof}
First, suppose that $R_1,R_2$ are dualizing complexes over $A$. Then $D_A(R_1)$ and $D_A(R_2)$ are tilting complexes, so that $P = D_A(R_1)\otimes^{\mrm{L}}_A D_A(R_2)$ is also a tilting complex. Hence, 
\[
D_A(P) = \mrm{R}\opn{Hom}_A(P,R_A) \cong \mrm{R}\opn{Hom}_A(P,A)\otimes^{\mrm{L}}_A R_A.
\]
But $\mrm{R}\opn{Hom}_A(P,A)$ is also tilting, so that $D_A(P) \cong R_1\otimes^{!}_A R_2$ is a dualizing complex, so by Theorem \ref{thm:hocharetwist}, the complex
\[
\mrm{R}\opn{Hom}_{A\otimes^{\mrm{L}}_{\k} A}(A,R_1\otimes^{\mrm{L}}_{\k} R_2)
\]
is also a dualizing complex. Next, let $R$ be a dualizing complex over $A$. Let $R' = A\otimes^{\mrm{L}}_{A\otimes^{\mrm{L}}_{\k} A} \mrm{R}\opn{Hom}_{\k}(R,R_A) \cong \opn{Hom}^!_A(R,R_A)$. A similar calculation to the above now shows that $R'$ is a dualizing complex, and that 
\[
\mrm{R}\opn{Hom}_{A\otimes^{\mrm{L}}_{\k} A}(A,R\otimes^{\mrm{L}}_{\k} R') \cong R \otimes^{!}_A R' \cong R_A. 
\]
By Corollary \ref{cor:ass}, the operation $\mrm{R}\opn{Hom}_{A\otimes^{\mrm{L}}_{\k} A}(A,-\otimes^{\mrm{L}}_{\k}-)$ is associative.
It follows that $\mathcal{D}_A$ is an abelian group. It is clear that the map $D_A(-):\mathcal{D}_A \to \opn{DPic}(A)$ is bijective (the inverse map is also $D_A$). To see that it is a group map, simply note that 
\begin{eqnarray}
D_A(\mrm{R}\opn{Hom}_{A\otimes^{\mrm{L}}_{\k} A}(A,R_1\otimes^{\mrm{L}}_{\k} R_2)) \cong
D_A(R_1\otimes^{!}_A R_2) =\nonumber\\ 
=D_A( D_A( D_A(R_1)\otimes^{\mrm{L}}_A D_A(R_2) ) ) \cong D_A(R_1)\otimes^{\mrm{L}}_A D_A(R_2)\nonumber.
\end{eqnarray}
Finally, if $f:A\to B$ is a $\k$-algebra map, then it is well known that $f^{!}$ maps $\mathcal{D}_A$ to $\mathcal{D}_B$, and Corollary \ref{cor:commute} shows that it is a homomorphism.
\end{proof}

\begin{rem}
Returning to the symmetric monoidal structure on $\mrm{D}^{+}_{\mrm{f}}(\opn{Mod} A)$ described in Remark \ref{rem:monoidal} above, one may easily see that the units of this structure are precisely the dualizing complexes. Thus, the fundamental result that $f^{!}$ carries dualizing complexes to dualizing complexes is expressed by the fact that monoidal functors carry unit elements to unit elements. Similarly, the fact that $f^{!}$ carries rigid dualizing complexes to rigid dualizing complexes (\cite[Theorem 0.1]{YZ2}, \cite[Corollary 3.3.5]{AIL2}) can now be expressed simply by saying that monoidal functors preserve the monoidal identity.
\end{rem}

We end this note with a couple remarks about possible generalizations of the above theory.

\begin{rem}
In \cite[Corollary 6.5]{AILN}, there is a global version of the reduction formula for derived Hochschild cohomology under the additional assumption that the given scheme is flat over the base. A similar result for derived Hochschild homology is shown in \cite[Theorem 4.1.8]{ILN}. Using these results, all results of this note immediately generalize to the global case of schemes, under the additional assumption that they are flat over $\k$.
\end{rem}

\begin{rem}
Another possible generalization is relaxing the assumptions on $\k$. As a first step, one can relax the regularity assumption and assume instead that $\k$ is Gorenstein. Rigid dualizing complexes still exist (see \cite[Theorem 8.5.6]{AIL1}), so most of the above will still make sense and will be true, under the additional assumption that all algebras are of finite flat dimension over $\k$. Going further, we can simply assume that $\k$ is a noetherian ring. Then, in the (possible) absence of dualizing complexes, we may use instead the notion of a relative dualizing complex (see \cite[Section 1]{AILN}). Again, we will have to assume that all algebras are of finite flat dimension over $\k$, and further, to have the biduality isomorphism of \cite[Theorem 1.2]{AILN}, we must assume that all complexes involved are also of finite flat dimension over $\k$.
\end{rem}

\textbf{Acknowledgments.}
The author would like to thank Professor Joseph Lipman and Professor Amnon Yekutieli for some useful suggestions. The author also wish to thank the Department of Mathematics at the Weizmann Institute of Science, where this work was carried out.

\end{document}